\documentclass[10pt]{amsart}
\usepackage{amsfonts,amsmath,amsthm,amssymb}
\usepackage[all]{xy}
\usepackage[pagebackref]{hyperref}
\numberwithin{equation}{section}

\begin{document}

\newtheorem{theorem}{Theorem}[section]
\newtheorem{lemma}[theorem]{Lemma}
\newtheorem{proposition}[theorem]{Proposition}
\newtheorem{corollary}[theorem]{Corollary}

\theoremstyle{definition}
\newtheorem{definition}[theorem]{Definition}
\newtheorem{example}[theorem]{Example}

\theoremstyle{remark}
\newtheorem{remark}[theorem]{Remark}

\newenvironment{magarray}[1]
{\renewcommand\arraystretch{#1}}
{\renewcommand\arraystretch{1}}

\newcommand{\mapor}[1]{\smash{\mathop{\longrightarrow}\limits^{#1}}}
\newcommand{\mapin}[1]{\smash{\mathop{\hookrightarrow}\limits^{#1}}}
\newcommand{\mapver}[1]{\Big\downarrow
\rlap{$\vcenter{\hbox{$\scriptstyle#1$}}$}}

\newcommand{\Set}{\mathbf{Set}}
\newcommand{\Art}{\mathbf{Art}}

\renewcommand{\bar}{\overline}
\newcommand{\de}{\partial}
\newcommand{\debar}{{\overline{\partial}}}
\newcommand{\per}{\!\cdot\!}
\newcommand{\Oh}{\mathcal{O}}
\newcommand{\sA}{\mathcal{A}}
\newcommand{\sB}{\mathcal{B}}
\newcommand{\sC}{\mathcal{C}}
\newcommand{\sF}{\mathcal{F}}
\newcommand{\sG}{\mathcal{G}}
\newcommand{\sH}{\mathcal{H}}
\newcommand{\sI}{\mathcal{I}}
\newcommand{\sL}{\mathcal{L}}
\newcommand{\sM}{\mathcal{M}}
\newcommand{\sP}{\mathcal{P}}
\newcommand{\sU}{\mathcal{U}}
\newcommand{\sV}{\mathcal{V}}
\newcommand{\sX}{\mathcal{X}}
\newcommand{\sY}{\mathcal{Y}}

\newcommand{\Aut}{\operatorname{Aut}}
\newcommand{\Mor}{\operatorname{Mor}}
\newcommand{\Def}{\operatorname{Def}}
\newcommand{\Hom}{\operatorname{Hom}}
\newcommand{\HOM}{\operatorname{\mathcal H}\!\!om}
\newcommand{\DER}{\operatorname{\mathcal D}\!er}
\newcommand{\Der}{\operatorname{Der}}
\newcommand{\End}{{\operatorname{End}}}
\newcommand{\END}{\operatorname{\mathcal E}\!\!nd}
\newcommand{\Image}{\operatorname{Im}}
\newcommand{\coker}{\operatorname{coker}}
\newcommand{\tot}{\operatorname{tot}}
\newcommand{\Diff}{\operatorname{Diff}}
\newcommand{\Coder}{\operatorname{Coder}}
\newcommand{\MC}{\operatorname{MC}}
\newcommand{\TW}{\operatorname{TW}}

\renewcommand{\Hat}[1]{\widehat{#1}}
\newcommand{\dual}{^{\vee}}
\newcommand{\desude}[2]{\dfrac{\de #1}{\de #2}}

\newcommand{\A}{\mathbb{A}}
\newcommand{\N}{\mathbb{N}}
\newcommand{\R}{\mathbb{R}}
\newcommand{\Z}{\mathbb{Z}}
\newcommand{\C}{\mathbb{C}}
\renewcommand{\L}{\mathbb{L}}
\newcommand{\proj}{\mathbb{P}}
\newcommand{\K}{\mathbb{K}\,}

\newcommand{\contr}{{\mspace{1mu}\lrcorner\mspace{1.5mu}}}

\newcommand{\bi}{\boldsymbol{i}}
\newcommand{\bl}{\boldsymbol{l}}


\title[Formality of Koszul brackets]{Formality of Koszul brackets and deformations of holomorphic Poisson manifolds}
\author{Domenico Fiorenza}
\address{\newline
Universit\`a degli studi di Roma ``La Sapienza'',\hfill\newline
Dipartimento di Matematica \lq\lq Guido
Castelnuovo\rq\rq,\hfill\newline
P.le Aldo Moro 5,
I-00185 Roma, Italy.}
\email{fiorenza@mat.uniroma1.it}
\urladdr{www.mat.uniroma1.it/people/fiorenza/}
\author{Marco Manetti}
\email{manetti@mat.uniroma1.it}
\urladdr{www.mat.uniroma1.it/people/manetti/}

\date{May 14, 2012}

\begin{abstract} We show that if a generator of a differential Gerstenhaber algebra satisfies certain Cartan-type identities, then the corresponding Lie bracket is formal.
Geometric examples include the shifted  de Rham complex of a Poisson manifold and
the subcomplex of differential forms on a symplectic manifold vanishing on a Lagrangian submanifold, endowed with the Koszul bracket.
As a corollary we generalize a recent 
result by Hitchin on deformations of holomorphic Poisson
manifolds.
\end{abstract}

\maketitle

\section*{Introduction}
In the paper \cite{koszul}, Jean-Louis Koszul considered a graded commutative algebra $A=\oplus_{p\in\Z} A^p$ with unit $1\in A^0$ and a differential operator $\bl\colon A\to A$ of second order, of odd degree $k$, such that $\bl(1)=0$ and $\bl^2=0$. Then he proved that the bracket generated by $\bl$,
\[ [,]_{\bl}\colon A^p\times A^q\to A^{p+q+k},\quad
[a,b]_{\bl}:= (-1)^{p}(\bl(ab)-\bl(a)b)-a\bl(b),\]
satisfies both Poisson and Jacobi identities and then induces on $A$ what is nowadays called a  structure of Batalin-\!Vilkovisky algebra \cite{CFL,getzler94}.

Koszul's construction applies in particular when $A$ is the de Rham complex of a differentiable manifold and $\bl=\bl_{\pi}$ is the Lie derivative with respect to  a Poisson bivector 
$\pi$; the degree of $\bl_{\pi}$ is $-1$ and then it induces in particular a structure of differential graded Lie  algebra (DGLA) on the de Rham complex, with degrees shifted by 1.

However, in this case we have by Cartan formulas that 
$\bl_{\pi}=[\bi_{\pi},d]$ and $[\bl_{\pi},\bi_{\pi}]=0$, where $d$ is the de Rham differential and $\bi_{\pi}$ is the  interior product by $\pi$. As a consequence of this fact, Sharygin and Talalaev \cite{ST08} obtain that such a differential graded Lie algebra
is quasi-isomorphic to an abelian DGLA. In this paper we will reobtain Sharygin and Talalaev formality theorem as a particular case of a more general statement. More precisely, we will show that if $(A,d)$ is a differential graded commutative algebra endowed with a degree $-2k$ second order differential operator $\bi:A\to A$ such that $\bi(1)=0$, then $(A,d)$ carries a natural Gerstenhaber algebra structure whose underlying DGLA is formal as soon as the differential operator $\bl=[\bi,d]$ is such that $[\bl,\bi]$ is a second order differential operator. Moreover our proof also shows that the same conclusion holds for certain subcomplexes of $A$;  a remarkable example is the subcomplex of differential forms on a symplectic manifolds vanishing on a Lagrangian submanifold.  

The formality of the shifted de Rham complex is particularly relevant and useful in formal deformation theory, in view of the fact that  quasi-isomorphic DGLAs have isomorphic associated deformation functors.  As an application  we obtain an extension, and a new proof, of a recent result by Hitchin \cite{hitchin} on  deformations of holomorphic Poisson manifolds: let $\pi$ be a holomorphic Poisson structure on a compact complex  
manifold $X$, and let $\pi^{\#}\colon \Omega^1_X\to \Theta_X$ be the corresponding anchor map;  if the natural map $H^2_{dR}(X,\mathbb{C})\to H^2(X,\Oh_X)$  is surjective,
then for every closed $(1,1)$ form $\omega$, the class $[\pi^{\#}(\omega)]\in H^1(X,\Theta_X)$ is tangent to a deformation of $X$ over a smooth basis.

It's a pleasure for us
to thank Marco Gualtieri, Donatella Iacono, Rita Pardini, Bruno Vallette and Marco Zambon for useful and interesting discussions on the topics of this paper. We are indebted with Florian Sch\"{a}tz for pointing out  our attention to the paper \cite{ST08}.

\subsection*{Glossary and background}

We assume that the reader is familiar with the basic theory of differential graded Lie algebras (DGLA), $L_{\infty}$-algebras and related notions 
(Maurer-Cartan equation, gauge action etc.). For an introduction to these topics we refer to \cite{FMM,LodayVallette,manRENDICONTI,FreZamb} and references therein. 
We shall say that a DGLA is homotopy abelian if it is isomorphic to an abelian DGLA in the homotopy category, i.e. if it is quasi-isomorphic to an abelian differential graded Lie algebra.

For the clarity of exposition, we will distinguish the various Lie brackets appearing in this paper according to the following notation:

\begin{itemize}
\item[{$[\;,\;]$}:] The graded commutator bracket, i.e. $[a,b]=ab-(-1)^{\bar{a}\;\bar{b}}ba$, where $\bar{a}$ is the degree of a homogeneous element $a$;
\item[{$[\;,\;]_{SN}$}:] The Schouten-Nijenhuis bracket on polyvector fields; 
\item[{$[\;,\;]_{\pi}$}:] The Koszul bracket associated to a tangent bivector field $\pi$. 
\end{itemize}

Given a graded vector space $V$  we will denote by $V[k]$ the 
$k$-fold desuspension of $V$:  $V[k]^i=V^{k+i}$.

\bigskip
\section{Review of Koszul brackets}

Let $X$ be a smooth differentiable manifold, $T_X$ its tangent bundle and $(A_X,d)=
(\oplus_p \Gamma(\bigwedge^p T_X^*),d)$ its de Rham complex.
For every $\eta\in \Gamma(\bigwedge^p T_X)$ we denote by 
\[ \bi_\eta\colon A_X^q\to A_X^{q-p},\qquad \bi_{\eta}(\alpha)=\eta\contr\alpha,\qquad \text{ the interior product by }\eta ,\]
\[ \bl_\eta\colon A_X^q\to A_X^{q-p+1},\qquad \bl_\eta=[\bi_{\eta},d],\qquad \text{ the Lie derivative}.\]
Recall that for $p=1$ the operator $\bi_{\eta}$ is a derivation of $A_X$ and 
$\bi_{\eta\wedge\mu}=\bi_{\eta}\circ\bi_{\mu}$.

Everyone is familiar with Cartan's formulas \cite{cartan50,ginzburg,periods}:
\begin{equation*}
[\bl_{\eta},d]=0,\quad [\bi_{\eta},\bi_{\mu}]=0,\quad [\bl_{\eta},\bi_{\mu}]=\bi_{[\eta,\mu]_{SN}},\quad 
[\bl_{\eta},\bl_{\mu}]=\bl_{[\eta,\mu]_{SN}},
\end{equation*}
where $[,]_{SN}$ is the Schouten-Nijenhuis bracket on polyvector fields.

\begin{definition}[{\cite[pag. 266]{koszul}}]\label{def.koszulmagribracket}
The \emph{Koszul bracket} associated to a  tangent bivector field 
$\pi\in \Gamma(\bigwedge^2 T_X)$ is the bilinear map 
$[,]_{\pi}\colon \bigwedge^2A_X[1]\to A_X[1]$ defined as
\[
[\alpha,\beta]_{\pi}=
(-1)^{p}
(\bl_{\pi}(\alpha\wedge\beta)-\bl_{\pi}(\alpha)\wedge\beta)-\alpha\wedge\bl_{\pi}(\beta),\qquad \alpha\in A_X^p,\beta\in A_X.\]
\end{definition}

Using the relation $\bl_{\pi}=[\bi_{\pi},d]$ we may write, for $\alpha\in A_X^p$, $\beta\in A_X$,
\[
[\alpha,\beta]_{\pi}=(-1)^{p}(\bi_\pi d(\alpha\wedge\beta)-d\bi_{\pi}(\alpha\wedge\beta)
+d(\bi_{\pi}(\alpha))\wedge\beta-
\bi_\pi(d\alpha)\wedge\beta)-\alpha\wedge\bi_{\pi}(d\beta)-\alpha\wedge d(\bi_{\pi}(\beta)),\]
and therefore the Koszul bracket of two closed forms is exact.

The restriction of the bracket to $A_X^1=\Gamma(T^*_X)$, also known as the \emph{Magri bracket} \cite{KSM90,KS08}, can be conveniently described in terms of 
the morphism of vector bundles
\[ \pi^{\#}\colon T_X^*\to T_X,\]
called \emph{anchor map}, defined by the formula 
\[ \bi_{\pi^{\#}(\alpha)}(\beta)=\bi_{\pi}(\alpha\wedge \beta),\qquad \forall\; \alpha,\beta\in T_X^*.\]

In fact it is well known, and in any case easy to prove, that for $\alpha,\beta\in A_X^1$ we have
\[ \bi_{\pi^{\#}(\alpha)}(d\beta)=\bi_\pi(\alpha\wedge d\beta)-\alpha\wedge\bi_{\pi}(d\beta)\]
and (see e.g. \cite{KSM90,Xu97,ginzburg,KS08})
\[ [\alpha,\beta]_{\pi}=\bl_{\pi^{\#}(\alpha)}(\beta)-
\bl_{\pi^{\#}(\beta)}(\alpha)-d\bi_{\pi}(\alpha\wedge \beta)=
\bi_{\pi^{\#}(\alpha)}(d\beta)-\bi_{\pi^{\#}(\beta)}(d\alpha)+d\bi_{\pi}(\alpha\wedge \beta).
\]

Assume now that $\pi\in \Gamma(\bigwedge^2 T_X)$ is a \emph{Poisson structure}: 
this means that $[\pi,\pi]_{SN}=0$. By Cartan formulas this implies  that
\[[\bl_{\pi},\bi_{\pi}]=0,\qquad \bl_{\pi}^2=\frac{1}{2}[\bl_{\pi},\bl_{\pi}]=
\frac{1}{2}[[\bl_{\pi},\bi_{\pi}],d]=0,\]
and these conditions ensure (see e.g. \cite{koszul,brylinski,Xu97} and \cite[Lemma 6.3.4]{ginzburg}) that the Koszul bracket satisfies Jacobi identity and therefore that the triple $(A_X,d,[,]_{\pi})$ is a differential Gerstenhaber algebra with an exact generator $\bl_{\pi}$. 
These properties will be also reproved in this paper as a byproduct of our computations.

Another important fact, that we will use in Section~\ref{sec.deformationpoisson}, is that 
the Poisson structure $\pi$ gives a Lie algebroid structure on the cotangent bundle $T_X^*$ \cite{Xu97,ginzburg,KS08}; in particular the anchor map $\pi^{\#}$ is a Lie morphism between the sheaf of 1-forms, endowed with Koszul-Magri bracket, and the sheaf of tangent vector fields \cite[equation 3.3]{koszul}.

Having in mind application to deformation theory, in this paper we are mainly interested to the differential graded Lie algebra $(A_X[1],d,[,]_\pi)$. In particular, we will be concerned with extensions and generalizations the following formality theorem.

\begin{theorem}[Sharygin-Talalaev {\cite{ST08}}]\label{theorem.magri-koszul-is-formal} In the notation above, 
if $\pi$ is a  Poisson structure on $X$,  then 
$(A_X[1],d, [,]_{\pi})$ is a \textbf{formal} differential graded Lie algebra.
\end{theorem}

We recall that a differential graded Lie algebra is called formal if it is quasi-isomorphic to its cohomology. Since the Koszul bracket is trivial in the de Rham cohomology of $X$, the formality of the DGLA $(A_X[1],d, [,]_{\pi})$  is equivalent to claiming that it is quasi-isomorphic to an abelian differential graded Lie algebra.

We will recover Sharygin-Talalaev formality theorem in Section 4, as a corollary of a more general statement involving a differential graded commutative algebra $(A,d)$ equipped with an even degree second order differential operator $\bi\colon A\to A$ with $\bi(1)=0$ and such that $[[\bi,d],\bi]$ is a second order differential operator, too.

\bigskip

\section{A simple formality criterion for DGLA}

Given a graded vector space $V$ on a characteristic $0$ field $\K$, we will denote by    
$\overline{S(V)}=\oplus_{i\ge 1}\bigodot^i V$ the graded symmetric coalgebra cogenerated by $V$. 
Denoting by 
\[D(V)=\Hom^*_{\K}(\overline{S(V)},V)=\prod_{i\ge 0} D_i(V),\quad\text{where}\quad  
D_i(V)=\Hom^*_{\K}({\bigodot}^{i+1} V,V),\]
the composition on the right with the natural projection 
$\overline{S(V)}\to \bigodot^{i+1} V$ give an inclusion $D_i(V)\subset D(V)$, while  the composition on the left with the  natural projection 
$\overline{S(V)}\to V$ gives an isomorphism of graded vector spaces (see e.g. \cite{K,manRENDICONTI})
\[ \Coder^*_{\K}(\overline{S(V)})\xrightarrow{\sim} D(V).\]
By the inverse isomorphism, an element $g$ in $D_n(V)$ corresponds to the coderivation
\[
a_0\odot a_1\odot\cdots \odot a_{n+m}\mapsto \sum_{\sigma}\varepsilon(\sigma)
g(a_{\sigma(0)},\ldots,a_{\sigma(m)})\odot a_{\sigma(m+1)}\odot\cdots\odot a_{\sigma(n+m)},
\] 
where $\varepsilon(\sigma)$ is the Koszul sign and the sum is carried over all the $(m+1,n)$-unshuffles $\sigma$. The graded vector space $\Coder^*_{\K}(\overline{S(V)})$ is a linear subspace of the graded associative algebra $\End^*_{\K}(\overline{S(V)})$ of linear endomorphisms of $\overline{S(V)}$, which is closed under the graded commutator bracket; hence $D(V)$ inherits a natural graded Lie algebra structure.
A simple computation shows that
for $f\in D_n(V)$ and $g\in D_m(V)$ we have
\[ [f,g]=f\bullet g-(-1)^{\bar{f}\;\bar{g}}g\bullet f\in D_{n+m}(V)\]
where 
\[ f\bullet g(a_0,\ldots,a_{n+m})=\sum_{\sigma}\varepsilon(\sigma)
f(g(a_{\sigma(0)},\ldots,a_{\sigma(m)}),a_{\sigma(m+1)},\ldots,a_{\sigma(n+m)}),\]
To prevent a possible misleading, let us explicitly remark that the pre-Lie operation $\bullet$ on $D(V)$ is \emph{not} associative.
Notice that the induced  bracket on the graded Lie subalgebra $D_0(V)$ is the same as the graded commutator bracket on $\Hom^*_{\K}(V,V)$.

\bigskip

Recall that $L_\infty$ structures on the graded vector space $V[-1]$ are the degree 1 elements  $\partial$ in $D(V)$   such that $[\partial,\partial]=0$; following \cite{K}, an  $L_{\infty}$ structure $\partial$ is called \emph{linear}  if $\partial\in  D_0(V)$.

If $(V,d)$ is a chain complex, then we can look at $(V,d)$ as a linear $L_\infty$-algebra, and so at $d$ as an $L_\infty$ structure on $V$. Using $d$ (seen as a coderivation) to ``translate the origin'' 
in $\Coder^*_{\K}(\overline{S(V)})$, we have that $L_\infty$ structures on $V$ can be seen as the degree 1 coderivations $\xi$ on $\overline{S(V)}$ such that $(d+\xi)^2=0$. This is conveniently rewritten as the Maurer-Cartan equation for the DGLA $\Coder^*_{\K}(\overline{S(V)})$:
\[
\delta \xi+\frac{1}{2}[\xi,\xi]=0,
\]
where $\delta$ is the adjoint of $d$ seen as a coderivation.

For any degree zero coderivation $R\in D_{>0}(V)=\prod_{i>0}D_i(V)$, the exponential $e^R$ is a well defined element in the graded associative algebra $\End^*_{\K}(\overline{S(V)})$, and it is immediate to see that, since $R$ is a coderivation, $e^R$ is actually a graded coalgebra automorphism of $\overline{S(V)}$ with inverse $e^{-R}$. Moreover, in the graded  associative algebra $\End^*_{\K}(\overline{S(V)})$ we have,  for any solution $\xi$ of the Maurer-Cartan equation in $D(V)$, 
\[
e^{R}(d+\xi)e^{-R}=d+e^R*\xi,
\]
where $\ast$ denote the gauge action in $D(V)$ (see, e.g. \cite{manettiseattle}):
\[ e^R*\xi=\xi+\sum_{n=0}^\infty\frac{(\mathrm{ad}_R)^n}{(n+1)!}([R,\xi]+[R,d]).\]
In particular, for any degree zero coderivation $R\in D_{>0}(V)$
the coderivation $\xi_R=e^R*0$ defines an $L_\infty$-algebra structure isomorphic (via $e^R$) to a linear one: $e^R d\, e^{-R}=d+\xi_R$.

\begin{remark}
The isomorphism $e^R$ can be conveniently written in terms of an operadic ``forest formula''. Namely, the $\Hom_{\K}( \bigodot^{m}V, \bigodot^{n}V)$-component of $e^R$ can be written as a weighted sum over oriented forests with $n$ roots and $m$ leaves, and whose internal $k$-valent vertices are decorated by the $\Hom_{\K}( \bigodot^{k}V, V)$-component of $R$. As usual in this kind of formulas, the weights are given by the (inverse of the) cardinality of the automorphism groups of the forests.
\end{remark}

\begin{theorem}\label{thm.main-theorem}
Let $(A,d)$ be a chain complex, let $R\in \mathrm{Hom}^{-2k}_{\K}(A\odot A,A)$ considered as a degree zero element of $D(A[2k])$, let $Q=[R,d]$, and let
\begin{equation}\label{equ.brackettone}
[a,b]_Q=(-1)^{\bar{a}}Q(a,b)
\end{equation}
the degree zero bracket on $A[2k-1]$ induced by $Q$ via decalage.
 If $[R,Q]=0$, then the bracket~\eqref{equ.brackettone} gives a formal homotopy abelian DGLA structure on 
 $(A[2k-1],d)$. More precisely the exponential of the coderivation $R$ is an  
$L_{\infty}$-isomorphism between the DGLA
$(A[2k-1],d,0)$ and $(A[2k-1],d,[,]_Q)$.
\end{theorem}
\begin{proof}
Since $R$ is a degree zero element of $D_{>0}(A[2k])$ and $[R,[R,d]]=0$, we have
\[
e^R d\, e^{-R}=d+e^R\ast 0=d+\sum_{n=0}^\infty\frac{(\mathrm{ad}_R)^n}{(n+1)!}([R,d])=d+Q.
\]
Hence the two DGLA $(A[2k-1],d,0)$ and  $(A[2k-1],d,[,]_Q)$ have isomorphic Bar constructions, i.e. they are isomorphic as $L_{\infty}$-algebras. Therefore,   
according to the Bar-Cobar resolution \cite{LodayVallette}, they are quasi-isomorphic as differential graded Lie algebras.
\end{proof}

The above theorem is one of the possible formality criteria and find
application only in some particular cases, for instance for the Koszul brackets. 
The reader may find similar results in \cite{GPR11,Vallette} and \cite[Thm. 9.13]{manRENDICONTI}.

\section{Differential operators on  graded commutative algebras}

The theory of differential operators on  commutative rings (see e.g. \cite{coutinho,ginzburg}) extends without difficulties to the graded case. 
Let $A=\oplus A^i$ be a graded commutative algebra with unit $1\in A^0$ over a field $\K$ of characteristic 0. Every $a\in A$ is also considered as an element of $\Hom^*_{\K}(A,A)$ acting by left multiplication: 
\[ a\colon A\to A,\qquad a(b)=ab.\]
Denote by $[,]$ the graded commutator on $\Hom_{\K}^*(A,A)$ and by
\[\Diff_{k}(A)=\bigoplus_{n\in\Z}\Diff^n_{k}(A)\subset \Hom_{\K}^*(A,A)\] 
the graded subspace of differential operators of order $\le k$:  recall that $\Diff_{k}(A)$ is defined recursively by setting 
$\Diff_k(A)=0$ for $k<0$ and 
\[\Diff_k(A)=\{f\in \Hom_{\K}^*(A,A)\mid [f,a]\in \Diff_{k-1}(A)\; \forall a\in A\}\]
for $k\ge 0$. 

Moreover
\[ \Diff_k(A)\Diff_h(A)\subset\Diff_{h+k}(A),\qquad
[\Diff_k(A),\Diff_h(A)]\subset\Diff_{h+k-1}(A)\]
and therefore the space $\Diff(A)=\bigcup_{k}\Diff_k(A)$ of differential operators is  a Lie subalgebra of $\Hom_{\K}^*(A,A)$.

The differential operator of order $\le k$ are stable under scalar extension: if $f\in \Diff_k(A)$ and $B$ is a graded commutative algebra, then $f\otimes Id\in \Diff_k(A\otimes_{\K} B)$.

For a fixed even integer $2k$, let $V=A[2k]$, i.e. $V=\oplus_{i\in\Z}V^i$  with $V^i=A^{i+2k}$. According to the natural isomorphism $D_0(V)=\Hom_{\K}^*(V,V)=\Hom_{\K}^*(A,A)$ we may consider $\Diff(A)$ as a Lie subalgebra of $D(V)$.

Also, for every $n\ge 0$ consider the multiplication map
\[ \mu_n\colon A^{\odot n+1}\to A,\qquad \mu_n(a_0\odot\cdots\odot a_n)=a_0a_1\cdots a_n.\]
We shall look at $\mu_n$ as a degree $2kn$ element in $D_n(V)$, for  every $n\ge 0$.

\begin{remark} The Lie subalgebra of $D(A)$ generated by the operators $\mu_n$, $n\ge 0$, is isomorphic to the Lie algebra of polynomial vector fields on the affine line vanishing in the origin, with 
$\mu_n$ corresponding to $\dfrac{-t^{n+1}}{(n+1)!}\dfrac{d\;}{dt}$.
\end{remark}

\begin{lemma} For a linear map $f\in \Hom_{\K}^*(A,A)$ the following conditions are equivalent:
\begin{enumerate} 

\item $f\in \Der^*_{\K}(A)$,

\item $[f,\mu_n]=0$ for every $n>0$,

\item $[f,\mu_1]=0$.
\end{enumerate}
\end{lemma}

\begin{proof} For every $a,b\in A$ we have
\[ [f,\mu_1](a,b)=f(ab)-f(a)b-(-1)^{\bar{a}\;\bar{b}}f(b)a=f(ab)-f(a)b-(-1)^{\bar{a}\;\bar{f}}af(b)\]
and therefore $f$ is a derivation if and only if  $[f,\mu_1]=0$. 
The proof that if $f$ is a derivation then $[f,\mu_n]=0$ for every $n>0$ is easy and omitted.
\end{proof}

\begin{theorem}\label{thm.seventerms} For a linear map $f\in \Hom_{\K}^*(A,A)$ the following conditions are equivalent:
\begin{enumerate} 

\item \label{lem.seventerms1}  $f\in \Diff_2(A)$ and $f(1)=0$,

\item \label{lem.seventerms2}  $f$ satisfies the ``seven terms'' condition
\[f(abc)+f(a)bc+(-1)^{\bar{a}\;\bar{b}}f(b)ac+(-1)^{\bar{c}(\bar{a}+\bar{b})}f(c)ab=
f(ab)c+(-1)^{\bar{a}(\bar{b}+\bar{c})}f(bc)a+(-1)^{\bar{b}\bar{c}}f(ac)b,\]

\item \label{lem.seventerms3}  the bilinear form
$\Phi(a,b)=f(ab)-f(a)b-(-1)^{\bar{a}\bar{f}}a f(b)$ satisfies the Poisson identity
\[ \Phi(a,bc)=\Phi(a,b)c+(-1)^{(\bar{a}+\bar{f})\bar{b}}b\Phi(a,c),\]

\item \label{lem.seventerms4}  $[f,\mu_2]=[[f,\mu_1],\mu_1]$.
\end{enumerate}
\end{theorem}

\begin{proof} If $f\in \Diff_2(A)$ then $[[[f,a],b],c]=0$ 
for every $a,b,c\in A$ and if in addition $f(1)=0$ then also
\[
\strut[[[f,a],b],c](1)+f(1)abc=0
\]
for every $a,b,c\in A$. Expanding the above expression one finds the seven terms condition, hence  \eqref{lem.seventerms1} implies \eqref{lem.seventerms2}. That  \eqref{lem.seventerms2} implies \eqref{lem.seventerms3} is immediate. Next, the Poisson identity means that for every $a$ 
the operator $\Phi(a,-)$ is a derivation. Since $[f,a]=\Phi(a,-)+f(a)$, this implies that $[f,a]\in \Diff_1(A)$ for any $a$, and so $f\in \Diff_2(A)$. Moreover, by the Poisson identity again, $f(1)=-\Phi(1,1)=0$. This shows that \eqref{lem.seventerms3} implies \eqref{lem.seventerms1}. 
Finally, showing that \eqref{lem.seventerms4} is equivalent to \eqref{lem.seventerms2}
is tedious but straightforward. \end{proof}

\begin{definition} 
A linear map $f\colon A\to A$ will be called a 
\emph{quasi-Batalin-Vilkovisky} operator if satisfies any of the equivalent conditions of Theorem~\ref{thm.seventerms}.
\end{definition}

\begin{remark} 
The name quasi-Batalin-Vilkovisky operator is motivated from the fact 
\cite{koszul,getzler94} that a Batalin-Vilkovisky algebra may be defined as the data of a graded commutative algebra $A$ and a  quasi-BV operator $\Delta$ of  odd degree  such that $\Delta^2=0$.
\end{remark}

\bigskip
\section{Formality of Koszul brackets}

Throughout this section, $(A,d)$ will be a differential graded commutative algebra (with a differential $d$ of degree 1) over a field of characteristic 0 
and $\bi$ a quasi-BV operator on $A$ of \emph{even} degree $-2k$. We will write $\bl=[\bi,d]$; since $d$ is a derivation, also $\bl$ is a quasi-BV operator on $A$, of degree $-2k+1$.

\begin{lemma} In the notation above, assume that  also $[\bl,\bi]$ is a quasi-BV operator on $A$.
Let $R\in  \Hom^{-2k}_{\K}(A\odot A,A)$ and $Q\in \Hom^{-2k+1}_{\K}(A\odot A,A)$ be the 
bilinear operators defined respectively as
\[
R(a,b)=\bi(ab)-\bi(a)b-a\bi(b),
\]
and
\[
Q(a,b)=\bl(ab)-\bl(a)b-(-1)^{\bar{a}}a\bl(b).
\]
Then $Q=[R,d]$ and $[Q,R]=0$ in the graded Lie algebra $D(A[2k])$.
\end{lemma}

\begin{proof} 
It is immediate to check that in $D(A[2k])$ one has $Q=[R,d]$. Moreover, by definition, $R=[\bi,\mu_1]$ and $Q=[\bl,\mu_1]$. By assumption we have
\[
 [\bi,\mu_2]=[[\bi,\mu_1],\mu_1];\qquad [\bl,\mu_2]=[[\bl,\mu_1],\mu_1];\qquad [[\bl,\bi]\mu_2]=[[[\bl,\bi]\mu_1],\mu_1].
\]
The graded Jacobi identity gives
\[ [[[\bl,\bi],\mu_1],\mu_1]=[[[\bl,\mu_1],\bi],\mu_1]+[[\bl,[\bi,\mu_1]],\mu_1]=
[[[\bl,\mu_1],\mu_1],\bi]+2[[\bl,\mu_1],[\bi,\mu_1]]+[\bl,[[\bi,\mu_1],\mu_1]],\]
and
\[[[\bl,\bi],\mu_2]=[[\bl,\mu_2],\bi]+[\bl,[\bi,\mu_2]].\]
Therefore 
\[ 0=[[[\bl,\bi],\mu_1],\mu_1]-[[\bl,\bi],\mu_2]=2[[\bl,\mu_1],[\bi,\mu_1]]=2[Q,R].\]
\end{proof}

\begin{example}\label{example.poisson} Let $(A_X,d)$ be the de Rham complex of a manifold $X$. 
Given $\eta\in \Gamma(\bigwedge^p T_X)$ we have $\bi_\eta\in \Diff_p(A_X)$; moreover, 
$\bi_\eta\in \Diff_{p-1}(A_X)$ if and only if $\eta=0$. According to the formula
\[ [\bl_{\eta},\bi_{\beta}]=\bi_{[\eta,\beta]_{SN}}\] 
we have that, for  $\pi\in \Gamma(\bigwedge^2 T_X)$, the operators $\bi_{\pi}$ and 
$[\bl_{\pi},\bi_{\pi}]$ are quasi-BV if and only if 
$[\pi,\pi]_{SN}=0$, i.e. if and only if $\pi$ is a Poisson structure.
\end{example}

From Theorem \ref{thm.main-theorem} we therefore obtain
\begin{theorem}\label{thm.gerstenhaber}
In the notation above, assume that  also $[\bl,\bi]$ is a quasi-BV operator on $A$, 
and let $[\,,\,]_{\bl}$ be the degree $-2k+1$ bracket on $A$ defined by
\[
[a,b]_{\bl}=(-1)^{\bar{a}}(\bl(ab)-\bl(a)b)-a\bl(b).
\]
Then $(A,d,\cdot,[\,,\,]_{\bl})$ is a Gerstenhaber algebra, whose underlying DGLA $(A[2k-1],d,[\,,\,]_{\bl})$ is a homotopy abelian DGLA. If in addition $\bl^2=0$, then $(A,d,\cdot,[\,,\,]_{\bl},\bl)$ is a 
Batalin-Vilkovisky algebra.
\end{theorem}

\begin{proof}
The only thing to be checked is the Poisson identity for the bracket $[\,,\,]_{\bl}$; by Theorem~\ref{thm.seventerms}, this is equivalent to saying that $\bl$ is a quasi-BV operator.  
\end{proof}

\begin{remark} 
An alternative proof of the above theorem can be given using the 
the results of \cite{BDA}, where it is (implicitely) proved  that the series of higher Koszul brackets gives  a morphism of graded Lie algebras and then commutes with adjoint actions; 
this is essentially the argument used in \cite{CS}.
\end{remark}

\begin{example}\label{ex.formalitypoissonmanifolds}
An immediate application of the above theorem is the following refined version of Theorem \ref{theorem.magri-koszul-is-formal}.
Let $X$ be a smooth manifold, $(A_X,d,\wedge)$ be its de Rham algebra, and $[\,,\,]_\pi$ the Koszul bracket induced by a Poisson bivector field $\pi$. By Example \ref{example.poisson}, the operator $\bi_{\pi}$ satisfies the hypothesis of Theorem \ref{thm.gerstenhaber} and so 
$(A_X,d,\wedge,[\,,\,]_\pi)$ is a Gerstenhaber algebra whose underlying DGLA $(A_X[1],d,[\,,\,]_\pi)$ is homotopy abelian.  
\end{example}

\begin{corollary}\label{cor.subcomplex}
In the hypothesis of Theorem~\ref{thm.gerstenhaber}, let $B$ a differential graded linear subspace of $A$ which is closed under the bilinear operator $R=[\bi,\mu_1]$. Then $(B[2k-1],d,[\,,\,]_{\bl})$ is a formal DGLA. 
\end{corollary}
\begin{proof}
Since $B$ is closed under $R$, then $\overline{S(B)}$ it is preserved by $e^R$ and the proof of Theorem \ref{thm.main-theorem} applies. 
\end{proof}

\begin{example}
Let $A$ and $\bi$ as in Theorem \ref{thm.gerstenhaber}, then, for any $p_0\geq 2k$, the subcomplex $B=\bigoplus_{p\geq p_0}A^p$ satisfies the assumptions of Corollary \ref{cor.subcomplex}.
\end{example}

\begin{example}
Let $X$ be a symplectic manifold, and let $j\colon Y\hookrightarrow X$ be the inclusion of a Lagrangian submanifold. Then the differential ideal $B:=\ker j^*\subset A_X$ satisfies the assumptions of Corollary \ref{cor.subcomplex}; this immediately follows by the Lagrangian Neighborhood Theorem. In particular the Koszul bracket induces a homotopy abelian DGLA structure on the (shifted) complex of differential forms vanishing on $Y$.\end{example}

\begin{corollary} Let $A_{\bullet}$ be a cosimplicial commutative differential 
graded algebra and let  $\bi_{\bullet}\colon A_{\bullet}\to A_{\bullet}$ be a cosimplicial linear map such that
$\bi_n\colon A_n\to A_n$ satisfies the assumption of Theorem~\ref{thm.gerstenhaber} for every $n$. 
Then the totalization of the cosimplicial DGLA $(A_{\bullet}[2k-1],d_{\bullet},[\,,\,]_{{\bl}_{\bullet}})$ 
is a homotopy abelian differential graded Lie algebra. 
\end{corollary}

\begin{proof} Immediate consequence of the definition of totalization, see e.g. 
\cite{BGNT,FMcone,FMM, hinichdescent,IM10}, and the fact that differential operators of order $\le 2$ are 
stable under scalar extension.
\end{proof}

\bigskip
\section{An application to deformations of holomorphic Poisson manifolds}
\label{sec.deformationpoisson}

In this section we will denote by $X$ a compact complex manifolds, by $\Theta_X$ and $\Omega^1_X$ the sheaves of holomorphic vector fields and holomorphic  1-forms respectively, by $A_X^{p,q}$ the space of differentiable forms of type  $(p,q)$ and by $H^*_{dR}(X,\C)$ the de Rham cohomology of $X$.

A holomorphic Poisson structure on a complex manifold $X$ 
is a holomorphic tangent bivector field $\pi\in H^0(X,\bigwedge^2 \Theta_X)$ such that
$[\pi,\pi]_{SN}=0$. As in the differentiable case the Poisson structure induce 
both a Koszul bracket
\[ [\,,\,]_{\pi}\colon  A^{p,q}_X\times A^{r,s}_X\to A^{p+r-1,q+s}_X\]
and an anchor map $\pi^{\#}\colon \Omega^1_X\to \Theta_X$ 
which is a morphism of sheaves of Lie algebras.  

Denoting by 
$F^0_X\supset F^1_X\supset\cdots$ the Hodge filtration:
\[ F^i_X=\bigoplus_{p\ge i,q} A^{p,q}_X,\]
we have, by previous results that the DGLA  
$(F^0_X[1], d,[,]_{\pi})$ is quasi-isomorphic to an abelian DGLA and 
$(F^1_X[1], d,[,]_{\pi})$ is 
a differential graded Lie subalgebra. This is not sufficient to ensure the formality of 
$F^1_X[1]$, however we have:

\begin{lemma}\label{lem.abelianitafuno} Assume that the inclusion $F^1_X\hookrightarrow F^0_X$ is injective in cohomology (e.g. if $X$ is K\"{a}hler), then the DGLA $(F^1_X[1], d,[,]_{\pi})$ is quasi-isomorphic to an abelian DGLA.
\end{lemma}

\begin{proof} This is an easy consequence of the homotopy classification 
of DGLAs and $L_{\infty}$-algebras \cite{K}. Indeed, let $f\colon\mathfrak{g}\to\mathfrak{h}$ be a dgla morphism, with $H^*(f)$ injective and $\mathfrak{h}$ quasi-abelian. Then, by the fact that ${\mathfrak h}$ quasi-abelian, we have a zigzag of quasi-isomorphisms of DGLAs $\mathfrak{h}\xleftarrow{\sim}\mathfrak{k}\xrightarrow{\sim}V$, with $V$ a graded vector space (considered as a DGLA with trivial differential and bracket). Let the DGLA $\mathfrak{l}$ be the homotopy fiber product of $\mathfrak{g}$ with $\mathfrak{k}$ over $\mathfrak{h}$, let the graded vector space $W$ be the image of  $H^*(\mathfrak{l})\to H^*(\mathfrak{k})\cong H^*(V)=V$, and let $\pi_W\colon V\to W$ be a graded linear projection. Then we 
have a homotopy commutative diagram of DGLAs
\par
\[
\xy (0,0)*{{\mathfrak g}\,} ; (15,-6)*+{{\mathfrak h}} **\dir{-}
?>* \dir{>}
,(30,0)*+{{\mathfrak k}} ; (15,-6)*+{{\mathfrak h}} **\dir{-}
?>* \dir{>}
,(30,0)*+{{\mathfrak k}} ; (45,-6)*+{V} **\dir{-}
?>* \dir{>}
,(15,6)*+{{\mathfrak l}} ; (0,0)*+{{\mathfrak g}\,} **\dir{-}
?>* \dir{>}
,(15,6)*+{{\mathfrak l}} ; (30,0)*+{{\mathfrak k}} **\dir{-}
?>* \dir{>}
,(45,-6)*+{V} ; (60,-12)*+{W} **\dir{-}
?>* \dir{>}
,(24,-.6);(21,-3)**\dir{~}
,(39,-2.2);(36,-1.2)**\dir{~}
,(9,5.4);(6,3)**\dir{~}
,(8,-1)*{\scriptstyle{f}}
 \endxy
\]
whose homotopy commutative square on the left is a homotopy pullback. In particular, the morphism $\mathfrak{l}\to\mathfrak{g}$ is a quasi-isomorphism, and so the morphism $\mathfrak{l}\to\mathfrak{k}$ is injective in cohomology. Therefore, the composition  ${\mathfrak l}\to W$ is a quasi-isomorphism and we have the zigzag of quasi-isomorphisms $\mathfrak{g}\xleftarrow{\sim}\mathfrak{l}\xrightarrow{\sim}W$, with $W$ abelian.
\end{proof}

If we are interested to obstructions of lifting Maurer-Cartan elements, then the assumption of the Lemma~\ref{lem.abelianitafuno} can be relaxed. Denoting by $\Art$ the category of local Artinian $\C$-algebras, for any DGLA $L$  the associated deformation functor $\Def_L\colon \Art\to \Set$ 
is defined as: 
\[ \Def_{L}(C)=\frac{\{x\in L^1\otimes\mathfrak{m}_C\mid dx+\frac{1}{2}[x,x]=0\}}{\text{gauge equivalence}},\]   
where $\mathfrak{m}_C$ is the maximal ideal of $C$.
Among the basic facts about DGLA and associated deformation functors we have 
(see e.g. \cite{manRENDICONTI,manettiseattle} for proofs and more details):
\begin{enumerate}

\item quasi-isomorphic DGLAs have isomorphic associated deformation functors;

\item abelian DGLAs have unobstructed associated deformation functors;

\item if $L\to M$ is a morphism of DGLAs and $\Def_M$ is unobstructed, 
then the obstructions of $\Def_L$ are contained in the kernel of $H^2(L)\to H^2(M)$. 
\end{enumerate}

\begin{lemma}\label{lem.Foneunobstructed} Assume that the natural map 
$H^2_{dR}(X,\mathbb{C})\to H^2(X,\Oh_X)$ is surjective. Then the functor
$\Def_{F^1[1]}$ is unobstructed.
\end{lemma}

\begin{proof} Since $F^0_X[1]$ is quasi-isomorphic to an abelian DGLA, the functor $\Def_{F^0_X[1]}$ is unobstructed and therefore the obstructions of $\Def_{F^1_X[1]}$ are contained in the kernel of 
$H^2(F^1_X[1])\to H^2(F^0_X[1])$. Now
the exact sequence
\[\xymatrix{H^2(F^1_X)\ar[r]&H^2(F^0_X)\ar@{=}[d]\ar[r]&H^2(F_X^0/F_X^1)\ar[d]^{\cong}\ar[r]
&H^2(F^1_X[1])\ar[r]&H^2(F^0_X[1])\\
&H^2_{dR}(X,\mathbb{C})\ar[r]&H^2(X,\Oh_X)&&}
\]
implies that $H^2(F^1_X[1])\to H^2(F^0_X[1])$ is injective.  
\end{proof}

\begin{theorem}\label{thm.deformationpoisson}  Let $\pi$ be a holomorphic Poisson structure on a compact complex  
manifold $X$ such that the natural map $H^2_{dR}(X,\mathbb{C})\to H^2(X,\Oh_X)$ is surjective.
Then for every closed $(1,1)$ form $\omega$, the class $[\pi^{\#}(\omega)]\in H^1(X,\Theta_X)$ is tangent to a deformation of $X$ over a smooth basis.
\end{theorem} 

\begin{proof} Since $X$ is compact, it has a semiuniversal deformation; according to Artin's theorem on the solution of analytic equations \cite{Artin68} 
it is sufficient to prove that the class of $\pi^{\#}(\omega)$ extends to a formal deformation over $\C[[t]]$.
The anchor map $\pi^{\#}$, being holomorphic, extends to a morphism of differential graded Lie algebras
\[ \gamma\colon F^1_X[1]\xrightarrow{\text{projection}}A_X^{0,*}
(\Omega^1_X)\xrightarrow{\pi^{\#}}A_X^{0,*}
(\Theta_X).\]
The DGLA  $A_X^{0,*}(\Theta_X)$ is the Kodaira-Spencer algebra of $X$ and its associated deformation functor is isomorphic to the functor of infinitesimal deformations of $X$, see \cite{CCK} and references therein. According to Lemma~\ref{lem.Foneunobstructed} the functor $\Def_{F^1_X[1]}$ is unobstructed and therefore the class
\[
[\omega]\in H^1(F^1_X[1])\cong\Def_{F^1_X[1]}\left(\frac{\C[t]}{(t^2)}\right)\]
extend to an element of $\Def_{F^1_X[1]}(\C[[t]])$. This implies in particular that 
$\gamma([\omega])=[\pi^{\#}(\omega)]$ extends to a deformation  of $X$
over $\C[[t]]$.
\end{proof}

\begin{remark} Theorem~\ref{thm.deformationpoisson} has been recently proved by Hitchin \cite{hitchin} under the  assumption that either $X$ is K\"{a}hler or $H^2(X,\Oh_X)=0$.
The proof of Theorem~\ref{thm.deformationpoisson} also shows that the assumption $\de\omega=\debar \omega=0$ can be replaced by the existence of a form $\eta\in A^{2,0}_X$ such that 
$\de\eta=0$, $\debar\eta=\de\omega$ and $\debar\omega=0$: in fact, since
$\gamma(\omega)=\gamma(\omega-\eta)$ it is sufficient to consider the cohomology class
$[\omega-\eta]\in H^1(F^1_X[1])$ as a Maurer-Cartan element.
\end{remark}

\end{document}